\documentclass[12pt]{amsart}
\RequirePackage[utf8]{inputenc}
\usepackage{amsfonts,amssymb,amsmath,amscd,amsthm,amsbsy,hyperref,tikz, appendix}
\usepackage{blkarray}
\usepackage{setspace}
\usepackage{fix-cm}
\usepackage[paper=a4paper, dvips, top=2.5cm, left=1.8cm, right=1.8cm, foot=1.3cm, bottom=3.2cm]{geometry}
\usepackage{marginnote,letltxmacro}

\LetLtxMacro\mn\marginnote
\usepackage{enumerate}
\tikzset{node distance=2cm, auto}
\title{Diagonalizing the Ricci tensor}
\author{Anusha M. Krishnan}

\address{\begin{tabular}{l}
215 Carnegie Building, Department of Mathematics, Syracuse University, Syracuse,\\ NY 13244, USA \\[0.2cm]
\emph{E-mail address}: {\tt akrish03@syr.edu}
\end{tabular}
}

\date{\today}

\newcommand\numberthis{\addtocounter{equation}{1}\tag{\theequation}}
\newcommand{\g}{\mathrm{g}}

\newcommand{\R}{\mathbb{R}}
\newcommand{\C}{\mathbb{C}}
\newcommand{\Z}{\mathbb{Z}}

\newcommand{\B}{\mathcal{B}}

\renewcommand{\gg}{\mathfrak{g}}

\renewcommand{\k}{\mathfrak{k}}
\newcommand{\p}{\mathfrak{p}}
\newcommand{\h}{\mathfrak{h}}

\newcommand{\m}{\mathfrak{m}}
\newcommand{\n}{\mathfrak{n}}

\newcommand{\ddr}{\frac{\partial}{\partial r}}

\newcommand{\G}{\mathsf{G}}
\newcommand{\K}{\mathsf{K}}
\renewcommand{\H}{\mathsf{H}}

\newcommand{\SU}{\mathsf{SU}}
\newcommand{\SO}{\mathsf{SO}}
\renewcommand{\O}{\mathsf O}

\newcommand{\dd}{\mathrm{d}}
\newcommand{\ggz}{\g_{\rm GZ}}
\newcommand{\vspan}{\operatorname{span}}
\newcommand{\diag}{\operatorname{diag}}
\newcommand{\Span}{\operatorname{span}}
\newcommand{\Ric}{\operatorname{Ric}}
\newcommand{\Ad}{\operatorname{Ad}}
\newcommand{\ad}{\operatorname{ad}}

\newtheorem{theorem}{Theorem}

\newtheorem{proposition}[theorem]{Proposition}
\newtheorem{remark}[theorem]{Remark}
\newtheorem{mainthm}{\sc Theorem}

\newtheorem{maincor}[mainthm]{\sc Corollary}
\newtheorem*{mainconj}{\sc Conjecture}

\newtheorem*{maindefn}{\sc Definition}
\newtheorem{example}[theorem]{Example}

\numberwithin{equation}{section}
\numberwithin{theorem}{section}

\begin{document}
\begin{abstract}
We show that a basis of a semisimple Lie algebra of compact type, for which any diagonal left-invariant metric has a diagonal Ricci tensor, is characterized by the Lie algebraic condition of being ``nice''. Namely, the bracket of any two basis elements is a multiple of another basis element. This extends the work of Lauret and Will \cite{lw13} on nilpotent Lie algebras. The result follows from a more general characterization for diagonalizing the Ricci tensor for homogeneous spaces. Finally, we also study the Ricci flow behavior of diagonal metrics on cohomogeneity one manifolds.
\end{abstract}
\maketitle
\reversemarginpar
\section{Introduction}
Given a Riemannian manifold $(M, \g)$, its Ricci tensor $\Ric_\g$ is also a symmetric 2-tensor, and hence \emph{locally} (i.e. at each point) there is a basis that diagonalizes $\Ric_\g$ at that point. In this article we study the problem of diagonalizing $\Ric_\g$ \emph{globally}, in a particular sense that we will describe below. Diagonalizing the Ricci tensor is helpful in studying the Einstein equation (see \cite{da09}), the prescribed Ricci curvature equation, and the Ricci flow on a homogeneous space or on a cohomogeneity one manifold.

In particular, we are interested in diagonal metrics on closed cohomogeneity one manifolds under the Ricci flow. The Ricci flow is the geometric PDE
\begin{equation}\label{eq:RF}
  \begin{aligned}
  \frac{\dd \g}{\dd t} &= -2\Ric_\g,\,\,\,  \g(0) &= \g_0
 \end{aligned}
\end{equation}
for evolving in time a given Riemannian metric $\g_0$ on a manifold $M$. A cohomogeneity one manifold $M$ is a manifold with an action by a Lie group $\G$ so that the generic orbit of the group action has codimension $1$. By a diagonal metric on a cohomogeneity one manifold, we mean a metric that is diagonal with respect to a basis consisting of Killing vector fields of the action of $\G$ along a geodesic orthogonal to all orbits. See \cite{ak04}, \cite{aik15}, \cite{iks16} and \cite{bk16} where the Ricci flow on compact cohomogeneity one manifolds is used (sometimes implicitly) to study questions about singularity formation and curvature evolution under the flow.

As we will see, Ricci-diagonality for a cohomogeneity one manifold is equivalent to Ricci-diagonality for a principal orbit, which is a homogeneous space $\G\cdot p \cong \G/\H$. Therefore, initially we focus our attention on the case where $M$ is a Lie group or a homogeneous space with an invariant metric. In those cases, the metric on $M$ is completely determined by the inner product on a single tangent space. For a Lie group $\G$, the tangent space at the identity element, $T_e \G$, can be identified with the Lie algebra $\gg$, and we make the following definition.

\begin{maindefn}
 A basis $\B$ for a Lie algebra $\gg$ is said to be \emph{stably Ricci-diagonal} if any diagonal left-invariant metric has diagonal Ricci tensor $\Ric_\g$. 
\end{maindefn}

We can similarly define stably Ricci-diagonal bases corresponding to a homogeneous space $\G/\H$. Which bases of a Lie algebra $\gg$ are stably Ricci-diagonal? The question has been answered for nilpotent Lie algebras: Lauret and Will proved \cite{lw13} that stably Ricci-diagonal bases are characterized by the Lie algebraic condition of being \textit{nice}. For the discussion in the present article, we redefine it as follows.
\begin{maindefn}
 A basis $\B = \{X_1, \cdots, X_n\}$ for a Lie algebra $\mathfrak{g}$ is said to be \emph{nice} if $[X_i , X_j]$ is always a scalar multiple of some element in the basis.
\end{maindefn}
Under the assumptions of our article this is equivalent to the definition in \cite{lw13} (see Remark \ref{rem:nice_lw}). The paper \cite{lw13} also provides examples of solvable Lie algebras for which \textit{stably Ricci-diagonal} and \textit{nice} are not equivalent. In this article we show that the two conditions are equivalent for semisimple Lie algebras of compact type (hence, for left-invariant metrics on compact Lie groups).

In fact, we prove a more general result for a compact homogeneous space $\G/\H$. The result for compact Lie groups will follow as a particular case. We now provide the notation needed to state this result. Let $Q$ be a bi-invariant metric on $\gg$. Let $\h\subset\gg$ be the Lie algebra of $\H$ and let $\h^\perp$ be a $Q$-orthogonal complement of $\h$ in $\gg$. The tangent space at $[\H]$, $T_{[\H]}(\G/\H)$, can be identified with $\h^\perp$. Under the adjoint action of $\H$ on $\n = \h^\perp$, we have the orthogonal decomposition into irreducible $\H$-modules $\n = \n_1 \oplus \cdots \oplus \n_l$. Let $\B= \{e_l\}_l$ be a $Q$-orthonormal basis for $\gg$ that respects this decomposition. The Lie algebra structure constants $\gamma_{ij}^k$ are defined via $[e_i, e_j] = \displaystyle\sum_k \gamma_{ij}^k e_k$, i.e. $\gamma_{ij}^k = Q([e_i, e_j], e_k)$. With this notation in hand, we make the following definition.
\begin{maindefn}
 A basis $\B = \{X_1, \cdots, X_n\}$ for $\mathfrak{h}^\perp$ is said to be \emph{nice} if $\displaystyle\sum_{\substack{e_\alpha \in \n_r\\ e_\beta \in \n_s}} \gamma_{\alpha\beta}^i\gamma_{\alpha\beta}^j = 0$ for all $r$, $s$, $i\neq j$ and $e_i \in \n_i$, $e_j \in \n_j$, where $\n_i$ and $\n_j$ are modules equivalent under the action of $\Ad(\H)$.
\end{maindefn}
In the case where the homogeneous space is simply a Lie group $\G$, the modules are one-dimensional vector spaces, all equivalent under the (trivial) action of $\H$, and each spanned by a single basis element. Therefore in the case of a Lie group, the above condition reduces to:
\begin{align*}
 \gamma_{rs}^i\gamma_{rs}^j = 0 \mbox{ for all $r,s$, and all $i\neq j$.}
\end{align*}
In other words, for any pair of indices $r$ and $s$, there is at most one index $k$ for which $\gamma_{rs}^k \neq 0$, so $[e_r, e_s]$ (if non-zero) is a multiple of a single basis element. Thus, this matches with the definition of nice basis for a Lie algebra $\mathfrak{g}$, stated earlier in this article.

Our first result characterizes stably Ricci-diagonal bases for a compact homogeneous space:
\begin{mainthm}\label{mainthm:sRd_homogeneous}
 Let $\G/\H$ be a compact homogeneous space, $Q$ a bi-invariant metric for $\G$, and $\B$ a $Q$-orthonormal basis for $\n = \h^\perp$. Then $\B$ is stably Ricci-diagonal if and only if $\B$ is a nice basis.
\end{mainthm}
We will also see that the equations only depend on $i, j$, i.e., are independent of the choice of basis elements $e_i\in \n_i$, $e_j\in \n_j$. As an immediate corollary, we have the following result which directly extends the work of \cite{lw13} to the compact setting:
\begin{maincor}\label{mainthm:nice_stably_Ric_diag}
 Let $\G$ be a compact Lie group with bi-invariant metric $Q$. Suppose $\B = \{e_i\}$ is a $Q$-orthonormal basis for $\gg$. Then $\B$ is stably Ricci-diagonal for left-invariant metrics on $\G$ if and only if $\B$ is a nice basis.
\end{maincor}

\begin{center}
 ***
\end{center}

In the second part of this article, we focus on cohomogeneity one manifolds. We are interested in the question of whether an invariant diagonal metric on a closed cohomogeneity one manifold ($M$, $\G$) remains diagonal (in the same basis) when evolved by the Ricci flow. When the answer is affirmative, it implies that there is a \emph{time-independent} frame that diagonalizes the metric restriction on each orbit. The time-independence is not guaranteed otherwise, isometry preservation notwithstanding. Preservation of diagonality has another important geometric consequence, namely that any curve transverse to all orbits which is a geodesic in the initial metric $\g_0$, will remain a geodesic (up to reparametrization) in the evolving metric $\g(t)$. In \cite{bk16}, a crucial step was to show that for the specific cohomogeneity one manifolds under consideration in that paper, diagonality of the metric is preserved under the Ricci flow.

The answer to this question a priori depends on the choice of basis used to describe the metric. As indicated earlier, we are considering a basis $\B$ of Killing vector fields along a minimal geodesic (denoted by $\gamma(r)$) orthogonal to the orbits of the action of $\G$, which parameterizes the orbit space $M/\G$. For an affirmative answer to the above question, at a minimum, the basis $\B' = \B \cup \{\ddr\}$ must be \textit{stably Ricci-diagonal}, meaning that every metric that is diagonal in this basis, also has a Ricci tensor that is diagonal in the same basis along $\gamma$. Otherwise the Ricci flow equation implies that the metric acquires off-diagonal terms to first order in time. It is worth noting that a basis $\B' = \B \cup \{\ddr\}$ for a cohomogeneity one manifold is stably Ricci-diagonal if and only if $\B$ is a stably Ricci-diagonal basis for the principal orbit $\G/\H$ (see Proposition \ref{propn:ric_homog}).

However, even if the basis is known to be stably Ricci-diagonal, it is not clear that the flow preserves diagonality of the metric since we cannot rule out the possibility of the metric acquiring off-diagonal components at a slower rate. Nevertheless, it seems natural to make the following conjecture:
\begin{mainconj}
 Let $\B'$ be a stably Ricci-diagonal basis for the cohomogeneity one manifold $(M, \G)$, and let $\g_0$ be a metric on $M$ that is diagonal with respect to $\B'$. Then the Ricci flow evolving metric $\g(t)$ is also diagonal in the basis $\B'$.
\end{mainconj}

The reason this is a non-trivial question is that an affirmative answer is equivalent to the existence of solutions to a degenerate parabolic system of coupled PDEs in the space ($r$) and time ($t$) variables with overdetermined boundary conditions. See \cite{kr19} for a more detailed discussion on this topic. Thus, this question is distinct from (and more difficult than) the similar question for homogeneous metrics. For the analogous result in the homogeneous setting, it is sufficient that the Ricci tensor of a diagonal metric also be diagonal in the same basis, and the conclusion follows from the existence and uniqueness theorem for ODEs.

We prove that the above conjecture holds for a special class of cohomogeneity one manifolds. To describe this class, we need to introduce some notation. Let $(M, \G)$ be a cohomogeneity one manifold with orbit space $M/\G \cong [0,L]$. Then $M$ admits a decomposition into disc bundles over the two non-principal orbits, $M = \G\times_{\K_-}D_-\cup \G\times_{\K_+}D_+$. Here $\H\subset\{\K_-, \K_+\}$ are (isotropy) subgroups of $\G$, and $D_\pm$ are Euclidean discs with $\partial D_\pm = S_\pm = \K_\pm/\H$. Conversely, any collection of groups $\H \subset \{\K_-, \K_+\} \subset \G$ where $\K_\pm/\H$ are spheres, gives rise to a cohomogeneity one manifold via the above union of disk bundles. We denote by $\h\subset \k_\pm\subset \gg$ the corresponding Lie algebras. Then the tangent space at a point $p$ in $M$ is identified with $\h^\perp \oplus \vspan\{\ddr\}$, where $r$ is a variable parameterizing the orbit space. We use $\B$ to denote a basis for $\h^\perp$ and define $\B' = \B \cup \{\ddr\}$.
\begin{mainthm}\label{mainthm:RF_SO}
 Let $M$ be a manifold with a cohomogeneity one action by $\G = \SO(n)$. Suppose that the isotropy groups $\H$, $\K_\pm$ are each products of block embeddings of $\SO(k)$'s. Let $\B$ be a basis for $\h^\perp$ that is a subset of the natural basis $\{E_{ij}\}$ of $\mathfrak{so}(n)$ and $\g_0$ a diagonal metric with respect to $\B'$. Then the Ricci flow solution starting at $\g_0$ is also diagonal with respect to $\B'$ as long as the flow exists.
\end{mainthm}
Here $\mathfrak{so}(n)$ is the Lie algebra of $\SO(n)$ and $E_{ij}$ is the matrix with a $1$ in the $(i,j)$ entry, a $-1$ in the $(j,i)$ entry and $0$'s in all other entries. As we will see, $\{E_{ij}\}_{1\leq i<j \leq n}$ is a nice basis for $\mathfrak{so}(n)$.

We remark that the main ingredient in the proof is the presence of ``extra'' discrete isometries for diagonal metrics on such manifolds. Additionally, the conclusion of Theorem \ref{mainthm:RF_SO} holds for a slightly larger class of manifolds, see Theorem \ref{thm:rf_diag_pres_so_disconnected}. These include, for example, cohomogeneity one actions on spheres and projective spaces.

As an application, in the next result we show that the techniques of the main theorem in \cite{bk16} extend to higher dimensions.
\begin{mainthm}\label{mainthm:sec_GZ}
 Let $M$ be a manifold with a  cohomogeneity one action by $\SO(n)$ with a group diagram where the groups $\H$, $\K_\pm$ are products of $\SO(k)$'s in  block embedding, and such that there are two singular orbits each of codimension two. Then $M$ admits a metric $\g_0$ such that $\sec_{\g_0} \geq 0$ and when evolved by the Ricci flow, $\g(t)$ immediately acquires some negatively curved $2$-planes.
\end{mainthm}
This paper is organized as follows. In Section \ref{sec:cohom1mfds} we give some background and explain our notation for invariant cohomogeneity one metrics. The relevant notation for homogeneous spaces is also presented implicitly in this section. In Section \ref{sec:nice_stably_Ric_diag} we prove Theorem \ref{mainthm:sRd_homogeneous}. In Section \ref{sec:examples} we provide examples and investigate the nice basis condition for some compact semisimple Lie algebras. Section \ref{sec:SO_grp_diagram} contains the proof of Theorem \ref{mainthm:RF_SO}. Finally, Theorem \ref{mainthm:sec_GZ} is proved in Section \ref{sec:appln}.
\subsection*{Acknowledgements} I am grateful to Wolfgang Ziller and Renato Bettiol for many detailed and helpful comments. I also thank Lee Kennard and William Wylie for their encouragement. I thank the anonymous referee for their useful comments and suggestions.
\section{Cohomogeneity one manifolds}\label{sec:cohom1mfds}
In this section, we recall the definition of a cohomogeneity one group action and describe the structure of invariant metrics on a cohomogeneity one manifold. For more details one may refer to \cite{gz00}, \cite{gz02}.
\subsection{Cohomogeneity one structure}
A Lie group $\G$ is said to act on a manifold $M$ with cohomogeneity one if the orbit space $M/\G$ is 1-dimensional (equivalently, if the generic orbits of the group action are codimension one hypersurfaces). If $M$ is compact, this implies that $M/\G$ is either an interval $[0,L]$ or a circle $S^1$. The former is guaranteed when the manifold is simply connected. We will assume from now on that $M/\G = [0,L]$. Let $\pi$ be the quotient map $M\rightarrow M/\G$. The generic orbits, i.e. $\pi^{-1}(r)$ for $r\in(0,L)$ are called principal orbits. The open set $M^0$ formed by the union of all the principal orbits is sometimes referred to as the principal part of $M$. The orbits $B_- = \pi^{-1}(0)$ and $B_+ = \pi^{-1}(L)$ are called singular orbits.

Pick any point $x_- \in B_-$ and let $\gamma(r)$ be a minimal geodesic normal to $B_-$, with $\gamma(0) = x_-$ and meeting the other singular orbit $B_+$ for the first time in $\gamma(L) = x_+$. Then $\gamma(r)$ for $r\in(0,L)$ parametrizes the orbit space. The isotropy group is the same group $\H\subset \G$ at all points $\gamma(r)$ with $0<r<L$ and $\H$ is called the principal isotropy group. The isotropy groups $\K_\pm$ at $x_\pm$ are the singular isotropy groups. Thus each principal orbit is isometric to a homogeneous space $\G/\H$ and the singular orbits $B_\pm$ are isometric to $\G/\K_\pm$ respectively.

By the Slice Theorem, $M$ is composed of disk bundles over the singular orbits $B_-$ and $B_+$, glued along their common boundary $\G/\H$. This also implies that $\K_\pm/\H$ are diffeomorphic to spheres $S^{l_\pm}$. The data $\H\subset \K_\pm\subset \G$ is called a group diagram, and determines the cohomogeneity one manifold up to equivariant diffeomorphism.
\subsection{Invariant metrics}
By symmetry, any invariant metric is completely determined by specifying it along $\gamma$. Thus a cohomogeneity one metric on the principal part of $M$ has the following form:
\begin{equation*}
 \begin{aligned}
  \g(r) = \dd r^2 + \g_r, \mbox{\,\, $r \in (0, L)$},
 \end{aligned}
\end{equation*}
where $\g_r$ is a one parameter family of homogeneous metrics on the fixed homogeneous space $\G/\H$. This metric extends across the singular orbits to yield a smooth metric on all of $M$ if and only if the metric and its derivatives satisfy certain differential conditions known as \textit{smoothness conditions} at the endpoints $r=0$ and $r=L$ (see \cite{vz18}).
\subsection{Diagonal metrics}
We will now explain more carefully what we mean by a diagonal metric on a cohomogeneity one manifold. Let $\H\subset \K_\pm\subset \G$ be the group diagram and let $\mathfrak{h}\subset \mathfrak{k}_\pm\subset \mathfrak{g}$ be the corresponding Lie algebras. Let $Q$ be a bi-invariant metric on $\mathfrak{g}$ and $\mathfrak{m}_\pm=\mathfrak{k}_\pm^\perp$, $\mathfrak{p}_\pm=\mathfrak{h}^\perp \cap \mathfrak{k}_\pm$ with respect to this metric. Thus $\mathfrak{g} = \mathfrak{h} \oplus \mathfrak{p}_- \oplus \mathfrak{m}_- = \mathfrak{h} \oplus \mathfrak{p}_+ \oplus \mathfrak{m}_+$.

Let $\{X_i\}_{i=1}^m$ be a $Q$-orthonormal basis for $\mathfrak{h}^\perp$ that respects the decompositions $\mathfrak{h}^\perp = \mathfrak{m}_+\oplus\mathfrak{p}_+ = \mathfrak{m}_-\oplus\mathfrak{p}_-$. The existence of such a basis is also an assumption on the group diagram! For instance, the group diagram described in the following example \emph{does not} admit such a basis.
\begin{example}
 Let $\{e_1, e_2, e_3\}$ be a $Q$-orthonormal basis for $\gg = \mathfrak{so}(3)$. Consider the group diagram whose corresponding Lie algebras are given by
 \begin{align*}
  \k_- &= \vspan\{e_1\};\,\,\,  \k_+ = \vspan\left\{e_1+e_2-2e_3\right\};\,\,\, \h = \{0\}\\
  \implies \p_- &= \vspan\{e_1\};\,\, \m_- = \vspan\{e_2, e_3\};\,\, \p_+ = \vspan\left\{e_1+e_2-2e_3\right\};\\\m_+ &= \vspan\left\{ e_1+e_2+e_3, e_1-e_2 \right\}
 \end{align*}
 A basis that respects the decomposition $\h^\perp = \m_-\oplus\p_-$ must have $v_1 = e_1$ as an element. A basis that respects the decomposition $\h^\perp = \m_+\oplus\p_+$ must have $v_2 = \frac{e_1+e_2-2e_3}{\sqrt{6}}$ as an element. Since $Q(v_1,v_2) \neq 0$, such a basis is not $Q$-orthogonal.
\end{example}
The vector space $\mathfrak{h}^\perp$ can be identified with the tangent space to $\G/\H$ at $[\H]$ in the following way. Let $\{X_i^*(r)\}_{i=1}^m$ be Killing vector fields along the curve $\gamma$, defined by
\begin{align*}
 X_i^*(r) = \frac{\dd}{\dd s}\exp(s\, X_i)\cdot\gamma(r)\big|_{s=0}
\end{align*}
Then $\{X_i^*(r)\}_{i=1}^m$ is a basis for $T_{[\H]}\G/\H$ at $\gamma(r) = [\H]$. Also, for $i=1,\,\cdots,\,m$, let $\omega_i$ be the 1-form along $\gamma$ dual to the vector field $X_i^*$. A diagonal metric is one which is of the form
\begin{equation*}
 \begin{aligned}
    \g(r) = h(r)^2 \dd r^2 + \sum_{i=1}^m f_i(r)^2 \omega_i^2, \text{  $ r\in (0,L)$}
 \end{aligned}
\end{equation*}
along a fixed geodesic orthogonal to all the orbits.
\begin{remark}
 The metric is not necessarily diagonal at points outside the geodesic $\gamma$. The value of $\, \g(X_i^*, X_j^*)$ at an arbitrary point of $M$ is determined by its value along $\gamma$, with the help of the group action. In particular,
 \begin{align*}
  \g(X_i^*, X_j^*)|_{g\H} = \g(\Ad_{g^{-1}}X_i^*, \Ad_{g^{-1}}X_j^*)|_{\H}
 \end{align*}
Since the metric on the homogeneous space $\G/\H$ is left-invariant but not necessarily bi-invariant, the Killing vector fields $X_i^*$ and $X_j^*$ for $i\neq j$ will in general not be orthogonal at points not on $\gamma$.
\end{remark}
In the next example, we provide a basis that respects the group diagram, but for which we can explicitly check that the basis is not nice and not stably Ricci-diagonal.
\begin{example}
 The Kervaire sphere $S^5$ has a cohomogeneity one action (see \cite{gvwz06}) with the following group diagram:
\begin{align*}
 \G &= \SO(2)\times \SO(3),\\ \K_- &= \SO(2) = (e^{-i\theta}, \diag(R(d\theta), 1)),\\ \K_+ &= \O(2) = (\det B, \diag(\det B, B)),\\ \H &= \Z_2 = \langle\, ( -1, \diag(-1, -1, 1) )\,\rangle,
\end{align*}
where $d$ is an odd integer. We select the following basis for $\gg$, which respects the inclusions $\h\subset\mathfrak{k}_\pm\subset\gg$ and is orthonormal in the natural bi-invariant metric on $\G$:
\begin{align*}
 X_1 = \frac{1}{d^2+1}(-I, dE_{12}), X_2 = \frac{1}{d^2+1}(dI,E_{12}), X_3 = (0, E_{13}), X_4 = (0, E_{23})
\end{align*}
Then, $[X_3, X_4] = -\frac{d}{d^2+1}X_1 - \frac{1}{d^2+1}X_2$, so this is not a nice basis.

This basis is also not stably Ricci-diagonal. Indeed, we can choose the metric such that at some point in the interior of the geodesic $\gamma$, the functions $f_i$ all have the same value. At such a point, \cite[Proposition 1.14]{gz02} implies $\Ric(X_1, X_2) = \frac{1}{2}\frac{d}{(d^2+1)^2} \neq 0$.
\end{example}
Sometimes we will also need the following notation. The isotropy group $\H$ acts on $\n = \h^\perp$ via the adjoint action, and we have $\n = \n_1 \oplus \cdots \oplus \n_l$ a sum of $Q$-orthogonal irreducible $\H$-modules. Then, by Schur's lemma, $\g|_{\n_i}$ is a multiple of $Q|_{\n_i}$, $\g|_{\n_i} = f_i(r)\cdot Q|_{\n_i}$ at $\gamma(r)$. It will usually be clear from context whether we are using a given index to denote an $\H$-module or an individual vector.
\section{Nice bases and stably Ricci-diagonal bases}\label{sec:nice_stably_Ric_diag}

In this section, we give a Lie-algebraic characterization for a basis to be stably Ricci-diagonal, proving Theorem \ref{mainthm:sRd_homogeneous}. First, we note some properties of the Lie algebra structure constants, which will enable us to simplify the expression for $\Ric_\g$.

Let $\G$ be a compact group with bi-invariant metric $Q$ and $\{e_l\}$ a $Q$-orthonormal basis for $\gg$. If $\gamma_{ij}^k = Q([e_i, e_j], e_k)$ then
 \begin{align*}
 \gamma_{ij}^k &= -\gamma_{ik}^j = -\gamma_{kj}^i = -\gamma_{ji}^k
 \end{align*}
since $\ad_X$ is skew-symmetric in $Q$. In particular, if any two indices are equal, then that structure constant is zero.

We refer to the formulae for Ricci curvature as derived in \cite[Proposition 1.14]{gz02}. The formulae there are for the Ricci curvature of a cohomogeneity one manifold, but it is easy to read off the Ricci curvature of a principal orbit/ homogeneous space by subtracting the second fundamental form contribution, i.e. any term involving derivatives. Using the skew-symmetry to collect terms and simplify, we write the expression for $\Ric_\g$ for a homogeneous space $\G/\H$ as follows:
\begin{equation}\label{eqn:ric_homog_biinv}
\begin{aligned}
 \Ric^{\G/\H}(e_i, e_j) = \sum_{r,s}\left( \frac{f_i^2f_j^2 - (f_r^2- f_s^2)^2}{4f_r^2f_s^2} \sum_{\substack{e_\alpha \in \n_r\\e_\beta \in \n_s}} \gamma_{\alpha\beta}^i\gamma_{\alpha\beta}^j \right)
\end{aligned}
\end{equation}
where $e_i\in \n_i$, $e_j\in \n_j$ and $i\neq j$. We omit the formula for the restriction of $\Ric_\g$ to a single $\H$-module $\n_i$, since irreducibility of $\n_i$ implies $\Ric_\g|_{\n_i}$ is diagonal.

With these preliminaries, we are now ready to show that a nice basis is stably Ricci-diagonal.
\begin{proposition}\label{propn:Ric-diag}
 Consider a homogeneous space $\G/\H$. Let $Q$ be a bi-invariant metric on $\G$ and let $\B = \{e_\alpha\}$ be a $Q$-orthonormal basis for $\h^\perp$ such that $\B$ is a nice basis. Then $\B$ is stably Ricci-diagonal.
\end{proposition}
\begin{proof}
 Suppose $\B$ is a nice basis. Let $\g$ be a $\B$-diagonal metric on $\G/\H$. If $\n_i$ and $\n_j$ are inequivalent modules then they are automatically orthogonal with respect to any invariant symmetric 2-tensor. In particular we have $\Ric(\n_i, \n_j) = 0$ whenever $\n_i$ and $\n_j$ are inequivalent modules.
 
 Now, suppose $\n_i$ and $\n_j$ are equivalent modules. By \eqref{eqn:ric_homog_biinv} and the nice basis condition, for $e_i \in \n_i$, $e_j \in \n_j$ with $i\neq j$, we have
 \begin{align*}
  \Ric^{\G/\H}(e_i, e_j) = \sum_{r,s}\left( \frac{f_i^2f_j^2 - (f_r^2- f_s^2)^2}{4f_r^2f_s^2} \cdot 0 \right) = 0.
 \end{align*}
 Hence $\B$ is stably Ricci-diagonal.
\end{proof}
We will now prove the converse.
\begin{theorem}\label{thm:nice_stably_Ric_diag}
 Let $\G/\H$ be a compact homogeneous space and $Q$ a bi-invariant metric on $\gg$. Let $\B = \{e_i\}$ be a $Q$-orthonormal basis for $\h^\perp$. Suppose that $\B$ is stably Ricci-diagonal. Then $\B$ is a nice basis.
\end{theorem}
\begin{proof}
 Let $\g$ be a $\B$-diagonal metric on $\G/\H$, given by $\g|_{\n_i} = f_i^2Q|_{\n_i}$.
  By \eqref{eqn:ric_homog_biinv}, the off-diagonal terms of Ricci are given by
  \begin{align}\label{eqn:Ric_cross}
   \Ric^{\G/\H}(e_i, e_j)
   = \sum_{r,s}\left( \frac{f_i^2f_j^2 - (f_r^2- f_s^2)^2}{4f_r^2f_s^2} \Upsilon_{rs}^{ij} \right),
  \end{align}
where we have set $\sum_{\substack{e_\alpha \in \n_r\\e_\beta \in \n_s}} \gamma_{\alpha\beta}^i\gamma_{\alpha\beta}^j = \Upsilon_{rs}^{ij}$ for convenience in the upcoming calculations. To show that $\B$ is a nice basis, it will suffice to show that $\Upsilon_{rs}^{ij} = 0$ whenever $i\neq j$.

If $\B$ is stably Ricci-diagonal then the right hand side of \eqref{eqn:Ric_cross} must equal $0$ for every choice of values $f_k$. This gives us equations the symbols $\Upsilon_{ab}^{cd}$ must satisfy, and choosing sufficiently many values of the constants $f_k$, we will show that $\Upsilon_{rs}^{ij} = 0$ whenever $i\neq j$. The skew-symmetry of structure constants implies that $\Upsilon_{rs}^{ij} = \Upsilon_{sr}^{ij}$ for all pairs $r,s$, and we will use this fact implicitly in the calculations below.

If we let $\g_0$ be the diagonal metric where $f_k = c$ for each $k$, then $\Ric_{\g_0}(e_i, e_j)=0$ implies
\begin{equation}\label{eq:hsum_zero1}
 \sum_{r,s} \Upsilon_{rs}^{ij} = 0.
\end{equation}
Next, fix one $r_0 \neq i, j$. Consider the metric $\g_1$ where $f_{r_0} = \sqrt{2}c$ and $f_k=c$ for all $k\neq r_0$. Then $\Ric_{\g_1}(e_i, e_j) = 0$ implies
\begin{align*}
 \frac{c^4 - (2c^2 - 2c^2)^2}{16c^4}\Upsilon_{r_0r_0}^{ij} + 2\sum_{s\neq r_0}\frac{c^4 - (2c^2 - c^2)^2}{8c^4}\Upsilon_{r_0 s}^{ij} &+ \sum_{\substack{r,s:\\ r,s \neq r_0}} \frac{c^4 - (c^2-c^2)^2}{4c^4} \Upsilon_{rs}^{ij} = 0\\
 \implies & \frac{1}{4}\Upsilon_{r_0r_0}^{ij} + \sum_{\substack{r,s:\\ r,s \neq r_0}}\Upsilon_{rs}^{ij} = 0 \numberthis \label{eq:hsum_zeroo}
\end{align*}
With the same fixed $r_0$, consider the metric $g_2$ where $f_{r_0} = \sqrt{3}c$ and $f_k=c$ for all $k\neq r_0$. Then $\Ric_{\g_1}(e_i, e_j) = 0$ implies
\begin{align*}
 \frac{c^4 - (3c^2 - 3c^2)^2}{36c^4}\Upsilon_{r_0r_0}^{ij} + 2\sum_{s\neq r_0}\frac{c^4 - (3c^2 - c^2)^2}{12c^4}\Upsilon_{r_0 s}^{ij} &+ \sum_{\substack{r,s:\\ r,s \neq r_0}} \frac{c^4 - (c^2-c^2)^2}{4c^4} \Upsilon_{rs}^{ij} = 0\\
 \implies & \frac{1}{9}\Upsilon_{r_0r_0}^{ij} - 2\sum_{s\neq r_0} \Upsilon_{r_0 s}^{ij} + \sum_{\substack{r,s:\\ r,s \neq r_0}}\Upsilon_{rs}^{ij} = 0 \numberthis \label{eq:hsum3_zeroo}
\end{align*}
Solving \eqref{eq:hsum_zero1}, \eqref{eq:hsum_zeroo} and \eqref{eq:hsum3_zeroo} simultaneously  we see that for each fixed $r_0\neq i,j$,
\begin{equation}\label{eq:hsum_zero2}
 \Upsilon_{r_0r_0}^{ij} = 0,\,\, \sum_{s\neq r_0} \Upsilon_{r_0 s}^{ij} = 0 \mbox{ and } \sum_{\substack{r,s:\\ r,s \neq r_0}}\Upsilon_{rs}^{ij} = 0 
\end{equation}
Now, fix a pair of \textit{distinct} indices $r_0\neq i,j$ and $r_1\neq i,j$. Consider the diagonal metric $\g_3$ where $f_{r_0} = f_{r_1} = \sqrt{2}c$ and $f_s = c$ for all $s\neq r_0, r_1$. Since by \eqref{eq:hsum_zero2}, $\Upsilon_{r_0r_0}^{ij} = 0$ and $\Upsilon_{r_1r_1}^{ij} = 0$, the condition $\Ric_{\g_3}(e_i,e_j)=0$ can be written as
\small
\begin{multline*}
\begin{aligned}
 2\sum_{s\neq r_0, r_1} \frac{f_i^2f_j^2 - (f_{r_0}^2-f_s^2)^2}{4f_{r_0}^2f_s^2} \Upsilon_{r_0s}^{ij} + 2\sum_{s\neq r_0, r_1} \frac{f_i^2f_j^2 - (f_{r_1}^2-f_s^2)^2}{4f_{r_1}^2f_s^2} \Upsilon_{r_1s}^{ij} + 2\frac{f_i^2f_j^2 - (f_{r_0}^2 - f_{r_1}^2)^2}{4f_{r_0}^2f_{r_1}^2}\Upsilon_{r_0r_1}^{ij}\\ + \sum_{\substack{r,s:\\r,s \neq r_0,r_1}}\frac{f_i^2f_j^2 - (f_r^2-f_s^2)^2}{4f_r^2f_s^2}\Upsilon_{rs}^{ij} = 0\\
\mbox{\normalsize Hence, } \sum_{s\neq r_0, r_1} \frac{c^4 - (2c^2-c^2)^2}{4c^4} \Upsilon_{r_0s}^{ij} + \sum_{s\neq r_0, r_1} \frac{c^4 - (2c^2-c^2)^2}{4c^4} \Upsilon_{r_1s}^{ij} + \frac{c^4 - (2c^2 - 2c^2)^2}{8c^4}\Upsilon_{r_0r_1}^{ij}\\ + \sum_{\substack{r,s:\\r,s \neq r_0,r_1}}\frac{c^4 - (c^2-c^2)^2}{4c^4}\Upsilon_{rs}^{ij} = 0\\
\end{aligned}
\end{multline*}
\normalsize
Therefore we obtain
\begin{equation}\label{eq:hsum_zero3}
 \Upsilon_{r_0r_1}^{ij} + 2 \sum_{\substack{r,s:\\ r,s \neq r_0,r_1}}\Upsilon_{rs}^{ij} = 0
\end{equation}
Equation \eqref{eq:hsum_zero1} can be written as follows:
\begin{align*}
 &\Upsilon_{r_0r_0}^{ij} + \Upsilon_{r_1r_1}^{ij} - 2\Upsilon_{r_0r_1}^{ij} + 2\sum_{s \neq r_0} \Upsilon_{r_0s}^{ij} + 2\sum_{s \neq r_1} \Upsilon_{r_1s}^{ij} + \sum_{\substack{r,s:\\ r,s \neq r_0, r_1}}\Upsilon_{rs}^{ij} = 0\\
 \mbox{and hence } & 0 + 0 - 2\Upsilon_{r_0r_1}^{ij} + 0 + 0 + \sum_{\substack{r,s:\\ r,s \neq r_0, r_1}}\Upsilon_{rs}^{ij} =  0 \mbox{ by $\eqref{eq:hsum_zero2}$}\\
 \mbox{Thus,}&\\
 & - 2\Upsilon_{r_0r_1}^{ij} + \sum_{\substack{r,s:\\ r,s \neq r_0, r_1}}\Upsilon_{rs}^{ij} =  0 \numberthis \label{eq:hsumzero}
\end{align*}
Combining \eqref{eq:hsumzero} with \eqref{eq:hsum_zero3} we obtain that for distinct $r_0, r_1$ with $r_0\neq i,j$ and $r_1\neq i,j$,
\begin{align}\label{eq:hsumr0r1}
 \sum_{\substack{r,s:\\ r,s \neq r_0, r_1}}\Upsilon_{rs}^{ij} =  0 \mbox{  and  }\Upsilon_{r_0r_1}^{ij}  = 0
\end{align}
In particular, $\Upsilon_{r_0r_1}^{ij}  = 0$ whenever $r_0\neq i,j$ and $r_1\neq i,j$.

Next, we deal with $\Upsilon_{rs}^{ij}$ where at least one of $r, s$ equals $i$ or $j$. Accordingly, consider the metric $\g_4$ where $f_i = f_j = c$ and $f_k = \sqrt{2}c$ for all $k\neq i, j$. Then $\Ric_{\g_4}(e_i, e_j) = 0$ implies
\begin{flalign*}
 &\frac{c^4 - (c^2-c^2)^2}{4c^4}\Upsilon_{ii}^{ij} + \frac{c^4 - (c^2-c^2)^2}{4c^4}\Upsilon_{jj}^{ij} + \frac{c^4 - (c^2-c^2)^2}{4c^4}2\Upsilon_{ij}^{ij} + \frac{c^4 - (2c^2 - c^2)^2}{8c^4} 2\sum_{r\neq i,j}\Upsilon_{ri}^{ij}\\
 & \hspace{4.9cm}+ \frac{c^4 - (2c^2 - c^2)^2}{8c^4} 2\sum_{r\neq i,j}\Upsilon_{rj}^{ij} + \frac{c^4 - (2c^2 - 2c^2)^2}{16c^4}\sum_{r,s \neq i,j}\Upsilon_{rs}^{ij} = 0\\
 &\implies \frac{1}{4}\Upsilon_{ii}^{ij} + \frac{1}{4}\Upsilon_{jj}^{ij} + \frac{1}{2}\Upsilon_{ij}^{ij} + 0 + 0 + 0 = 0 \hspace{0.5cm}\mbox{ (since $\Upsilon_{rs}^{ij} = 0$ when $r, s \neq i,j$)}\\
 &\implies \Upsilon_{ii}^{ij} + \Upsilon_{jj}^{ij} + 2\Upsilon_{ij}^{ij} = 0 \numberthis \label{eq:hsum0ij1}
\end{flalign*}
Next, we re-write equation \eqref{eq:hsum_zero1} as follows, to reflect the indices $i$ and $j$. 
\begin{align*}
 \Upsilon_{ii}^{ij} + \Upsilon_{jj}^{ij} + 2\Upsilon_{ij}^{ij} + 2\sum_{r\neq i,j}\Upsilon_{ri}^{ij} + 2\sum_{r\neq i,j}\Upsilon_{rj}^{ij} + \sum_{r,s \neq i,j}\Upsilon_{rs}^{ij} = 0
\end{align*}
Since $\Upsilon_{rs}^{ij} = 0$ when $r,s\neq i,j$, this reduces to:
\begin{align}\label{eq:hsum0ij2}
 \Upsilon_{ii}^{ij} + \Upsilon_{jj}^{ij} + 2\Upsilon_{ij}^{ij} + 2\sum_{r\neq i,j}\Upsilon_{ri}^{ij} + 2\sum_{r\neq i,j}\Upsilon_{rj}^{ij}  = 0
\end{align}
In subsequent calculations, we will use that $\Upsilon_{rs}^{ij} = 0$ when $r,s\neq i,j$ without explicitly stating it each time. Next, consider the metric $\g_5$ such that $f_i = c$, $f_j = \sqrt{2}c$, and $f_k = c$ for all $k\neq i,j$. Then $\Ric_{\g_5}(e_i, e_j) = 0$ implies
\begin{align*}
 \frac{2c^4 - (c^2 - c^2)^2}{4c^4} \Upsilon_{ii}^{ij} + \frac{2c^4 - (2c^2 - 2c^2)^2}{16c^4} \Upsilon_{jj}^{ij} + \frac{2c^4 - (c^2 - 2c^2)^2}{8c^4} 2\Upsilon_{ij}^{ij} + 2\sum_{r\neq i,j} \frac{2c^4 - (c^2 - c^2)^2}{4c^4} \Upsilon_{ri}^{ij}\\
 + 2\sum_{r\neq i,j} \frac{2c^4 - (c^2 - 2c^2)^2}{8c^4} \Upsilon_{rj}^{ij} + \sum_{r,s \neq i,j} \frac{2c^4 - (c^2 - c^2)^2}{4c^4} \Upsilon_{rs}^{ij} = 0\\
 \implies \frac{1}{2}\Upsilon_{ii}^{ij} + \frac{1}{8}\Upsilon_{jj}^{ij} + \frac{1}{4}\Upsilon_{ij}^{ij} + \sum_{r\neq i,j}\Upsilon_{ri}^{ij} + \frac{1}{4}\sum_{r\neq i,j}\Upsilon_{rj}^{ij}  = 0 \numberthis \label{eq:hsum0ij3}
\end{align*}
Similarly, if we consider the metric $\g_6$ where $f_i = \sqrt{2}c$, $f_j = c$, and $f_k = c$ for all $k\neq i, j$, we get:
\begin{align}\label{eq:hsum0ij4}
 \frac{1}{8}\Upsilon_{ii}^{ij} + \frac{1}{2}\Upsilon_{jj}^{ij} + \frac{1}{4}\Upsilon_{ij}^{ij} + \frac{1}{4} \sum_{r\neq i,j}\Upsilon_{ri}^{ij} + \sum_{r\neq i,j}\Upsilon_{rj}^{ij}  = 0
\end{align}
Also, the metric $\g_7$ where $f_i = c$, $f_j = \sqrt{2}c$ and $f_k = \sqrt{2}c$ for all $k\neq i, j$ yields
\begin{align}\label{eq:hsum0ij5}
 \frac{1}{2}\Upsilon_{ii}^{ij} + \frac{1}{8}\Upsilon_{jj}^{ij} + \frac{1}{4}\Upsilon_{ij}^{ij} + \frac{1}{4} \sum_{r\neq i,j}\Upsilon_{ri}^{ij} + \frac{1}{4}\sum_{r\neq i,j}\Upsilon_{rj}^{ij}  = 0
\end{align}
Solving \eqref{eq:hsum0ij1}, \eqref{eq:hsum0ij2}, \eqref{eq:hsum0ij3}, \eqref{eq:hsum0ij4} and \eqref{eq:hsum0ij5} together, we obtain
\begin{align}\label{eq:hsum0ijall}
 \Upsilon_{ii}^{ij} = 0,\,\,\, \Upsilon_{jj}^{ij} = 0,\,\,\, \Upsilon_{ij}^{ij} = 0,\,\,\, \sum_{r\neq i,j}\Upsilon_{ri}^{ij} = 0, \mbox{ and } \sum_{r\neq i,j}\Upsilon_{rj}^{ij}  = 0
\end{align}
It only remains to show that $\Upsilon_{r_0i}^{ij} = 0$ and $\Upsilon_{r_0j}^{ij} = 0$ when $r_0\neq i,j$. First note that by \eqref{eq:hsum_zero2},
\begin{align*}
 0 = \sum_{s\neq r_0} \Upsilon_{r_0s}^{ij} =  \Upsilon_{r_0i}^{ij} + \Upsilon_{r_0j}^{ij} + \sum_{s\neq r_0, i, j} \Upsilon_{r_0s}^{ij} = \Upsilon_{r_0i}^{ij} + \Upsilon_{r_0j}^{ij} + 0\\
 \implies \Upsilon_{r_0i}^{ij} + \Upsilon_{r_0j}^{ij} = 0 \numberthis \label{eq:hsum_r0ij}
\end{align*}
%
%
The fourth and fifth equations of \eqref{eq:hsum0ijall} can be written as:
\begin{align}
 \label{eq:hsum_a}\Upsilon_{r_0i}^{ij} + \sum_{r\neq r_0, i, j} \Upsilon_{ri}^{ij} = 0\\
 \label{eq:hsum_b}\Upsilon_{r_0j}^{ij} + \sum_{r\neq r_0, i, j} \Upsilon_{rj}^{ij} = 0 
\end{align}
Now, consider the metric $\g_8$ where $f_i = c$, $f_j = \sqrt{2}c$, $f_{r_0} = \sqrt{3}c$, and $f_k = c$ for all $k\neq i, j, r_0$. Then $\Ric_{\g_8}(e_i, e_j) = 0$ implies
\begin{align*}
 \frac{2c^4 - (3c^2 - c^2)^2}{12c^4} 2\Upsilon_{r_0i}^{ij} + \frac{2c^4 - (3c^2 - 2c^2)^2}{24c^4} 2\Upsilon_{r_0j}^{ij} + \frac{2c^4 - (c^2 - c^2)^2}{4c^4} 2\sum_{r\neq r_0, i, j} \Upsilon_{ri}^{ij}\\
 + \frac{2c^4 - (c^2 - 2c^2)^2}{8c^4} 2\sum_{r\neq r_0, i, j} \Upsilon_{rj}^{ij} = 0\\
 \implies -\frac{1}{3} \Upsilon_{r_0i}^{ij} + \frac{1}{12} \Upsilon_{r_0j}^{ij} + \sum_{r\neq r_0, i, j} \Upsilon_{ri}^{ij} + \frac{1}{4} \sum_{r\neq r_0, i, j} \Upsilon_{rj}^{ij} = 0 \numberthis \label{eq:hsumr0ijall}
\end{align*}
Finally, solving equations \eqref{eq:hsum_r0ij}, \eqref{eq:hsum_a}, \eqref{eq:hsum_b} and \eqref{eq:hsumr0ijall}, we obtain that when $r_0 \neq i, j$,
\begin{align}\label{eq:hsumr0ir0j}
 \Upsilon_{r_0i}^{ij} =0, \Upsilon_{r_0j}^{ij} = 0, \sum_{r\neq r_0, i, j} \Upsilon_{ri}^{ij} = 0, \mbox{ and } \sum_{r\neq r_0, i, j} \Upsilon_{rj}^{ij} = 0
\end{align}
To conclude, by \eqref{eq:hsum_zero2}, \eqref{eq:hsumr0r1}, \eqref{eq:hsum0ijall} and \eqref{eq:hsumr0ir0j}  we see that $\Upsilon_{rs}^{ij} = 0$ for all $r, s$ whenever $i\neq j$. Therefore $\B$ is a nice basis.
\end{proof} 
As a result, we have:
\begin{proof}[Proof of Theorem \ref{mainthm:sRd_homogeneous}]
 A direct consequence of Proposition \ref{propn:Ric-diag} and Theorem \ref{thm:nice_stably_Ric_diag}.
\end{proof}
\begin{remark}
 The number of equations in Theorem \ref{mainthm:sRd_homogeneous} can be reduced significantly. Indeed, let
 \begin{align*}
  B(X, Y) = \sum_{\substack{e_\alpha \in \n_r\\e_\beta \in \n_s}} Q([e_\alpha, e_\beta], X) Q([e_\alpha, e_\beta], Y) \mbox{ for all } X, Y\in \n.
 \end{align*}
 Then, since $\Ad_g$ preserves $Q$ and the Lie brackets, it follows that $B$ is $\Ad(\H)$-invariant. Thus $B|_{\n_i} = \lambda_i Q|_{\n_i}$ for some constant $\lambda_i$ and Theorem \ref{mainthm:sRd_homogeneous} reduces to a single equation for each $i, j, r, s$. Similarly, if the basis $\B$ respects the equivalence between two modules $\n_i$ and $\n_j$ then Theorem \ref{mainthm:sRd_homogeneous} reduces to $1$, $2$, or $4$ equations depending on whether the representations are orthogonal, complex or quaternionic.
\end{remark}
\begin{remark}\label{rem:nice_lw}
 In \cite{lw13}, the authors define a nice basis for a Lie algebra as one that satisfies the following two conditions:
 \begin{enumerate}
  \item $[X_i , X_j]$ is always a scalar multiple of some element in the basis.
  \item Two different brackets $[X_i , X_j]$, $[X_r , X_s]$ can be a nonzero multiple of the same $X_k$ only if $\{i, j\}$ and $\{r, s\}$ are disjoint.
 \end{enumerate}
  In this article, we always assume the basis is orthonormal with respect to a bi-invariant metric, hence we have the additional skew-symmetry of the structure constants. As a result, condition (2) above follows from condition (1). Indeed, if $[X_i, X_j]$ and $[X_i, X_s]$ are both non-zero multiples of the same basis element $X_k$, then $\gamma_{ij}^k \neq 0$ and $\gamma_{is}^k \neq 0$. By skew-symmetry of the structure constants, this implies $\gamma_{ik}^j \neq 0$ and $\gamma_{ik}^s \neq 0$. Then if $j \neq s$, it would contradict condition (1).
\end{remark}
We will now address Ricci-diagonality for cohomogeneity one manifolds.  We have:

\begin{proposition}\label{propn:ric_homog}
 Let $M$ be a cohomogeneity one manifold with principal part $M^0 = \G/\H \times (0, L)$ and let $\g = \dd r^2 + \g_r$ be a diagonal $\G$-invariant metric on $M^0$. Let $\B$ be a $Q$-orthonormal basis for $\h^\perp$. Then $\Ric_\g^M$ is diagonal in the basis $\B' = \B\, \cup \{\ddr\}$ at the point $\gamma(r)$ if and only if $\Ric_{\g_r}^{\G/\H}$ is diagonal in the basis $\B$.
\end{proposition}
\begin{proof}
 We can use the formulae for $\Ric_\g$ from \cite[Proposition 1.14]{gz02}.
 We will determine when the off-diagonal terms of $\Ric_\g$ are zero. Firstly, we note that for any diagonal metric $\g$, $\Ric_\g(\ddr, X) = 0$ for any $X$ tangent to the orbit $\G/\H$.
 
 Since the second fundamental form is diagonal, we have
 \begin{align*}
  \displaystyle \Ric_\g(e_i, e_j) = \Ric_{\g_r}(e_i, e_j)
 \end{align*}
 when $e_i\in \n_i$ and $e_j\in \n_j$ distinct $\H$-modules. Therefore, $\Ric_\g(e_i, e_j) = 0$ at $\gamma(r)$ if and only if $\Ric_{\g_r}(e_i, e_j) = 0$.

 If $e_1$, $e_2$ belong to the same $\H$-module $\n_i$, then
\begin{align*}
 \Ric_\g(e_1, e_2) &= \Ric_{\g_r}(e_1, e_2) + \left\{ \frac{-f_i'}{f_i}\sum_s \frac{f_s'}{f_s}\dim \n_s +  \frac{f_i'^2}{f_i^2} - \frac{f_i''}{f_i} \right\} f_i^2 Q(e_1, e_2)
\end{align*}
Thus $\Ric_\g(e_1, e_2) = 0$ if and only if $\Ric_{\g_r}(e_1, e_2) = 0$.
\end{proof}
Thus, for cohomogeneity one manifolds, the problem is reduced to understanding the stably Ricci-diagonal condition for a homogeneous metric on $\G/\H$, treated previously in this section.
\section{Examples}\label{sec:examples}
In the first part of this section, we examine some standard bases of semisimple Lie algebras (of compact type) to see if any of them are nice (and consequently stably Ricci-diagonal).
\subsection{$\mathfrak{so}(n)$} A basis for the Lie algebra $\mathfrak{so}(n)$ is given by $\{E_{ij}\}_{1\leq i<j\leq n}$. Recall that $E_{ij}$ is defined to be the $n\times n$ matrix with a $1$ in the $(i,j)$ entry, a $-1$ in the $(j,i)$ entry, and $0$ in all other entries. It is easy to see that the brackets are given by
 \begin{align*}
  [E_{ij}, E_{kl}] = \left\{
	\begin{array}{ll}
		0  & \mbox{if } \{i,j\} \cap \{k,l\} = \emptyset \\
		0 & \mbox{if } \{i,j\} = \{k,l\} \\
		E_{il} & \mbox{if } j = k \mbox{ and } i\neq l
	\end{array}
     \right.
 \end{align*}
The Lie bracket skew-symmetry and the fact that $E_{ji} = -E_{ij}$, yield all remaining brackets. From this it is clear that $\{E_{ij}\}_{i<j}$ is a nice basis and hence stably Ricci-diagonal.
\subsection{$\mathfrak{su}(n)$ with $n\geq 3$}
A basis for $\mathfrak{su}(n)$ is given by $\B = \{E_{pq}\}_{1\leq p<q\leq n}\cup\{F_{pq}\}_{1\leq p<q\leq n}\cup \{G_l\}_{1\leq l < n}$, where
\begin{itemize}
 \item $E_{pq}$ is as defined earlier.
 \item $F_{pq}$ is the matrix with $i$ in the $(p,q)$ and  $(q,p)$ entries and $0$ in all other entries.
 \item $G_l$ is the matrix with an $i$ in the $(l,l)$ entry, $-i$ in the $(l+1, l+1)$ entry, and $0$ in the other entries.
\end{itemize}
 Then this basis is not nice, since, for example, $[E_{13}, F_{13}] = G_1 + G_2$. But $\B$ is also not orthogonal w.r.t. bi-invariant metric! However, note that any basis $\B$ for $\mathfrak{su}(n)$ that contains $\{E_{pq}\}\cup\{F_{pq}\}$ cannot be a nice basis. Indeed, the set of brackets of basis elements must then contain the $^n C_2$ linearly independent elements $\{G_i\}_{i<j}$, and at most $n-1$ of these brackets can also be basis elements.
\subsection{$\mathfrak{su}(2)$} Consider the basis $\B = \{E_{12}, F_{12}, G_1\}$ where $E_{12}$, $F_{12}$, and $G_1$ are as above. Then $[E_{12}, F_{12}] = 2G_1$, $[F_{12}, G_1] = 2E_{12}$, and $[G_1, E_{12}] = 2F_{12}$, so $\B$ is a nice basis.
\subsection{$\mathfrak{sp}(n)$ with $n\geq 2$}
A basis for the Lie algebra $\mathfrak{sp}(n)$ is given by $\B = \{E_{pq}\}_{1\leq p<q\leq n}\cup\{F_{pq}\}_{1\leq p<q\leq n}\cup\{Y_{pq}\}_{1\leq p<q\leq n}\cup\{Z_{pq}\}_{1\leq p<q\leq n}\cup \{H_l\}_{1\leq l\leq n}\cup \{S_l\}_{1\leq l\leq n}\cup \{T_l\}_{1\leq l\leq n}$, where
\begin{itemize}
 \item $Y_{pq}$ is the matrix with $j$ in the $(p,q)$ and  $(q,p)$ entries and $0$ in all other entries.
 \item $Z_{pq}$ is the matrix with $k$ in the $(p,q)$ and  $(q,p)$ entries and $0$ in all other entries.
 \item $H_l$ is the matrix with an $i$ in the $(l,l)$ entry and $0$ in the other entries.
 \item $S_l$ is the matrix with an $j$ in the $(l,l)$ entry and $0$ in the other entries.
 \item $T_l$ is the matrix with an $k$ in the $(l,l)$ entry and $0$ in the other entries.
 \item $E_{pq}$ and $F_{pq}$ are as defined earlier.
\end{itemize}
Then $\B$ is orthonormal with respect to the bi-invariant metric, but it is not a nice basis either, since $[E_{12}, F_{12}] = \diag(i, -i, 0, \cdots, 0) = H_1 - H_2$.
\subsection{$\gg_2$} A basis for the Lie algebra $\gg_2\subset \mathfrak{so}(7)$ is given by
\begin{align*}
 X_1 = E_{12} - E_{47},\,\, X_2 = E_{12}+ E_{56},\,\, X_3 = E_{14} + E_{27},\,\, X_4 = E_{14} - E_{36},\,\, X_5 = E_{16} + E_{25}, \\ X_6 = E_{34} + E_{16},\,\,
 X_7 = E_{13} + E_{46},\,\, X_8 = E_{13} + E_{57},\,\, X_9 = E_{15} - E_{26},\,\,
 X_{10} = E_{15} - E_{37}, \\
 X_{11} = E_{17} - E_{24},\,\,
 X_{12} = E_{17} + E_{35},\,\,
 X_{13} = E_{23} + E_{67},\,\,
 X_{14} = E_{45} + E_{67}
\end{align*}
Then, $[X_1, X_6] = X_9 - X_{10}$, so this is not a nice basis. But it is also not orthonormal.\\
\newline
Thus, the basis $\{E_{ij}\}$ for $\mathfrak{so}(n)$ is the only nice basis amongst the above examples. It would be an interesting question for future investigation to find whether the compact semisimple Lie algebras apart from $\mathfrak{so}(n)$ do admit nice bases, and if so, to describe them.\\

We end this section by providing some examples to illustrate the nice basis condition for homogeneous spaces.
\subsection{$\SU(3)/\mathsf{T^2}$} Here $\mathsf{T^2}$ is the 2-torus embedded as $\diag(z, w, \overline {zw})$. Therefore, $\h = \Span\{G_1, G_2\}$ and $\h^\perp = \Span\{E_{12}, E_{13}, E_{23}, F_{12}, F_{13}, F_{23}\}$. Under the adjoint action of $\H = \mathsf{T^2}$, $\h^\perp$ splits as $\h^\perp = \n_1 \oplus \n_2 \oplus \n_3$, where $\n_1 = \Span\{E_{12}, F_{12}\}$, $\n_2 = \Span\{E_{13}, F_{13}\}$, and $\n_3 = \Span\{E_{23}, F_{23}\}$. The action of the element $\diag(z, w, \overline {zw})$ on $\n_1$, $\n_2$, and $\n_3$ (each considered as a copy of $\C$) is by multiplication by $z\overline{w}$, $z^2w$, and $zw^2$ respectively. Thus in this case, $\h^\perp$ is the sum of three inequivalent modules, and hence the nice basis condition is satisfied vaccuously, and the given basis is stably Ricci-diagonal.
\subsection{$\SU(3)/\mathsf{S^1}$} Here $\mathsf{S^1}$ is embedded as $\diag(z, 1, \overline{z})$. Therefore $\h = \Span\{G_1 + G_2\}$ and $\h^\perp = \Span\{E_{12}, E_{13}, E_{23}, F_{12}, F_{13}, F_{23}, G_1 - 2G_2\}$. Under the $\H$-action, we have $\h^\perp = \n_1 \oplus \n_2 \oplus \n_3 \oplus \n_4$, where $\n_1 = \Span\{E_{12}, F_{12}\}$, $\n_2 = \Span\{E_{13}, F_{13}\}$, $\n_3 = \Span\{E_{23}, F_{23}\}$, and $\n_4 = \Span\{G_1 - 2G_2\}$. Then $\n_1$ and $\n_3$ are equivalent modules where $\diag(z, 1, \overline{z})$ acts as multiplication by $z$, while the action on $\n_2$ is by multiplication by $z^2$, and $\n_4$ is a trivial module. As per the nice basis condition, we examine the quantity $\displaystyle\sum_{\substack{e_\alpha\in n_r\\ e_\beta\in \n_s}}\gamma_{\alpha\beta}^i\gamma_{\alpha\beta}^j$ for each $r, s$, and $e_i\in \n_1$, $e_j \in \n_3$, for example we take $e_i = E_{12}\in\n_1$ and $e_j = E_{23}\in\n_3$. It is easy to calculate and see that this quantity equals zero for each pair $r, s$. Hence the given basis for $\h^\perp$ is a nice basis (and hence stably Ricci-diagonal).
\section{Group diagrams involving standard basis of $\mathfrak{so}(n)$}\label{sec:SO_grp_diagram}
In this section, we restrict our attention to group diagrams where the groups $\G, \K_\pm, \H$ are all standard block embeddings of $\SO(k)$ or products of $\SO(k)$'s in $\SO(n)$. For this class of cohomogeneity one manifolds, there exists a nice (and hence stably Ricci-diagonal) basis adapted to the group diagram, namely a basis consisting of $E_{ij}$'s. The main result of this section is the proof of Theorem \ref{mainthm:RF_SO}. The proof proceeds via the observation that diagonal metrics on such a manifold have more symmetries than general $\SO(n)$-invariant metrics.
\begin{proposition}
\label{propn:conj_iso}
 Suppose $M$ is a cohomogeneity one manifold with group diagram where $\G = \SO(n)$ and $\H, \K_\pm$ are block embeddings of products of $\SO(k)$'s. Let $\B$ be a basis for $\h^\perp = \n$ consisting entirely of $E_{ij}$'s. Let $\g$ be a metric on $M$ that is diagonal in the basis $\B' = \B\cup \{ \ddr \}$. Then for each diagonal matrix $A \in \O(n)$, there is a map $\Phi_A : M\rightarrow M$ that is an isometry of $\g$.
\end{proposition}
\begin{proof} 
 We first observe that any isomorphism $\phi : \G \rightarrow \G$ such that $\phi(\K_\pm)\subset \K_\pm$ and $\phi(\H)\subset\H$ induces a diffeomorphism $\Phi$ of $M$. Indeed, fix a geodesic $\gamma$ for which the stabilizer groups are $\K_\pm$, $\H$. Any $p\in M$ is of the form $p = g_p\cdot \gamma(r)$ for some $r$ and $g_p\in\G$. Define $\Phi(g\cdot\gamma(r)) = \phi(g)\cdot\gamma(r)$. Then $\Phi$ is well-defined since $\Phi(gh\cdot\gamma(r)) = \phi(g)\phi(h)\cdot\gamma(r) = \phi(g)\cdot\gamma(r)$ for $0<r<L$ and similarly for $r = 0, L$. If $\dd\Phi$ preserves $\g$ at all points $\gamma(r)$ with $0<r<L$, then $\Phi$ is an isometry on the regular part of $M$ since $\G$ acts by isometries as well.
 
 In our case, we define $\phi_A : \G\rightarrow\G$ as conjugation by a diagonal element $A$ in $\O(n)$. Then $\phi_A$ preserves the groups $\K_\pm$, $\H$ and takes a basis vector $E_{ij}\in\B$ into $\pm E_{ij}$ and hence the induced diffeomorphism $\Phi_A :M\rightarrow M$ is an isometry in the diagonal metric.
\end{proof}
\begin{proposition}
\label{propn:inv_diag}
 Let $V$ be a subspace of $\mathfrak{so}(n)$ spanned by a subset of the $E_{ij}$'s. If $\g$ is a metric on $V$ which is invariant under $\Ad_A$ for all diagonal $A\in\O(n)$ then $\g$ is diagonal in the standard basis consisting of $E_{ij}$'s.
\end{proposition}
\begin{proof}
 For each pair of linearly independent elements $E_{ij}, E_{kl} \in \mathfrak{so}(n)$, there exists an element $A \in \O(n)$ such that $\Ad_A E_{ij} = E_{ij}$ and $\Ad_A E_{kl} = -E_{kl}$. Indeed, if $\{i,j\} \cap \{k,l\} = \emptyset$ then we can take $A$ to be the diagonal matrix with a $-1$ in the $(k,k)$ entry and $1$'s in the other diagonal entries. If $\{i,j\}\cap \{k,l\} = \{i\}$ then without loss of generality we are considering $E_{ij}$ and $E_{il}$, and we can take $A$ to be the diagonal matrix with $-1$ in the $(i,i)$ and $(j,j)$ entries and $1$'s in the other diagonal entries.
 
Then, invariance of the metric under $\Ad_A$ implies $\g(E_{ij}, E_{kl}) = 0$, since
 \begin{align*}
  \g(E_{ij}, E_{kl}) = \g(\Ad_A E_{ij}, \Ad_A E_{kl}) = \g(E_{ij}, -E_{kl}) = -\g(E_{ij}, E_{kl}).
 \end{align*}
\end{proof}
We are now ready to prove Theorem \ref{mainthm:RF_SO}.
\begin{proof}[Proof of Theorem \ref{mainthm:RF_SO}]
 Let $\g_0$ be a diagonal metric on $M$. By Proposition \ref{propn:conj_iso}, each diagonal matrix $A\in\O(n)$ yields an additional isometry $\Phi_A$ of $(M, \g_0)$.
 
 Since isometries are preserved under the Ricci flow \cite{ha82, ko10}, each $\Phi_A$ is an isometry of $(M, \g(t))$ as well. Now by Proposition \ref{propn:inv_diag}, any metric invariant under all of the $\Phi_A$'s must be diagonal. Thus $\g(t)$ is diagonal for each $t>0$ as well.
\end{proof}
In fact, the conclusion of Theorem \ref{mainthm:RF_SO} is also true for a slightly larger class of cohomogeneity one group diagrams.
\begin{theorem}
\label{thm:rf_diag_pres_so_disconnected}
 Suppose the group diagram of the cohomogeneity one manifold $M$ is such that
 \begin{enumerate}
  \item $\G = \SO(n)$
  \item $\H^0, \K_\pm^0$ are block embeddings of products of $\SO(k)$'s
  \item $\H \cong \H^0\rtimes B$, $\K_\pm \cong \K_\pm^0 \rtimes B$ where $B$ is a finite subgroup of $\G$ and $B$ acts on $\H^0$ or $\K_\pm^0$ by conjugation, i.e. $B$ is in the normalizer of $\H^0, \K_\pm^0$.
 \end{enumerate}
Then the diagonality of metrics with respect to the basis $\{E_{ij}\}$ is preserved under the Ricci flow.
\end{theorem}
\begin{proof}
 In the above situation, $M$ is the quotient by the right action of $B$ on the manifold $\widetilde{M}$ whose group diagram is $\H^0\subseteq \K_\pm^0\subseteq \G$. Therefore, any cohomogeneity one metric $\g$ on $M$ can be lifted to a cohomogeneity one metric $\tilde{\g}$ on $\widetilde{M}$ such that (at points on $\gamma$) $\tilde{\g}$ is invariant under the conjugation action by $B$. If $\g$ is diagonal then so is $\tilde{\g}$.
 
 Evolve the (diagonal) metric $\tilde{\g}$ via the Ricci flow. By Theorem
 \ref{mainthm:RF_SO}, the evolving metric $\tilde{\g}(t)$ on $M$ is diagonal. Since isometries are preserved under the Ricci flow, the (diagonal) evolving metric $\tilde{\g}(t)$ remains invariant under conjugation by $B$, and hence descends to a diagonal metric $\g(t)$ on $M$, such that $\g(t)$ also satisfies the Ricci flow equation with initial metric $\g$. By uniqueness of solutions to the Ricci flow, this shows that the metric on $M$ remains diagonal under the Ricci flow.
\end{proof}

The arguments above relied on the choice of basis elements $\{E_{ij}\}$ for the Lie algebras of $\H$, $\K$ and $\G$, which was made possible by assuming these groups were block embeddings of products of $\SO(k)$. Below we make the observation that if we want to work with connected Lie groups having basis consisting of some number of standard basis elements $\{E_{ij}\}$ in $\mathfrak{so}(n)$, then the only possibility for the groups $\H$ and $\K$ are block embeddings of products of $\SO(k)$'s.
\begin{proposition}
 Let $\k$ be a subalgebra of $\mathfrak{so}(n)$ that is generated (as an $\R$-vector space) by some subset $S$ of the $E_{ij}$'s. Then $\k$ is the direct sum of subalgebras of the form $\mathfrak{so}(k)$ in block embedding and hence $\K$ is a product of $\SO(k)$'s in block embedding.
\end{proposition}
\begin{proof}
 Write the set $S$ as the union of a finite number of sets $S_1$, $S_2$, $\cdots$, $S_m$ such that
 \begin{enumerate}
  \item $\{i, j\} \cap \{k, l\} = \emptyset$ whenever $E_{ij} \in S_\alpha$ and $E_{kl} \in S_\beta$ with $\alpha \neq \beta$.
  \item Each $S_\alpha$ is minimal (among subsets of $S$) with respect to property (1).
 \end{enumerate}
 Let $V_\alpha$ be the vector space generated by $S_\alpha$. Then $\k = \oplus_{\alpha} V_\alpha$ as a vector space. Further, each $V_\alpha$ is in fact a Lie subalgebra of $\k$. We can see this easily from the brackets among the $E_{ij}$'s; the brackets among basis elements in $V_\alpha$ cannot yield any indices that do not occur in $S_\alpha$, so $V_\alpha$ is closed under Lie brackets. Additionally, by (2) the brackets between $V_\alpha$ and $V_\beta$ are zero when $\alpha\neq \beta$. Thus it only remains to show that $V_\alpha \cong \mathfrak{so}(k_\alpha)$ in some block embedding. In fact, if $\Lambda = \{i_1, \cdots, i_l\}$ is the set of indices appearing in $S_\alpha$ then we will show that $V_\alpha$ is the $\mathfrak{so}(l)$ in the block embedding corresponding to the indices $i_1, \cdots, i_l$.
 
 If $S_\alpha$ has just one element $E_{ij}$ then clearly $V_\alpha \cong \R\cdot E_{ij} \cong \mathfrak{so}(1)$. If $S_\alpha$ has additional elements then without loss of generality there exists an index $k\neq i, j$ such that $E_{ik} \in S_\alpha$, since otherwise we could have split off $\{E_{ij}\}$ as another $S_i$, thus contradicting minimality of $S_\alpha$. Then $[E_{ij}, E_{ik}] = -E_{jk}$, so since $V_\alpha$ is closed under Lie brackets, $E_{jk}\in S_\alpha$. Hence $V_\alpha \supseteq \mathfrak{so}(3) = \vspan\{ E_{ij}, E_{jk}, E_{ik} \}$. Now if $S_\alpha$ has only three elements then $V_\alpha \cong \mathfrak{so}(3) \subseteq \mathfrak{so}(n)$ in block embedding in the indices $i, j, k$ and we are done. If $S_\alpha$ has additional elements then by minimality of $S_\alpha$, there exists an index $l\neq i, j, k$ such that $E_{il} \in S_\alpha$. Then $[E_{ij}, E_{il}] = -E_{jl}$, $[E_{jk}, E_{jl}] = -E_{kl}$ so since $V_\alpha$ is closed under brackets, $E_{il}, E_{jl}, E_{kl} \in S_\alpha$. Hence $V_\alpha \supseteq \mathfrak{so}(4) = \vspan\{ E_{ij}, E_{jk}, E_{ik}, E_{il}, E_{jl}, E_{kl} \}$. As before, if $S_\alpha$ has no more elements then we are done as $V_\alpha \cong \mathfrak{so}(4) \subseteq \mathfrak{so}(n)$ in block embedding in the indices $i, j, k, l$. Continuing this process, we see that
 \begin{enumerate}
  \item At each stage we obtain $V_\alpha \supset \mathfrak{so}(k)$ (in block embedding) for some $k$.
  \item This process must terminate since $V_\alpha$ is finite dimensional.
 \end{enumerate}
 Hence the termination of this process yields $V_\alpha \cong \mathfrak{so}(k)$ in some block embedding.
\end{proof}
\section{Instantaneous behaviour of Grove-Ziller metrics under Ricci flow}\label{sec:appln}
In this section, we use the result of the previous section to study the Ricci flow behaviour of certain cohomogeneity one $\sec\geq 0$ metrics. In particular, we prove Theorem \ref{mainthm:sec_GZ} from the Introduction, which extends the techniques of \cite{bk16} to higher dimensional cohomogeneity one manifolds. First, we derive the Ricci flow equations for a diagonal cohomogeneity one metric, assuming it evolves through other diagonal metrics. In the expressions below, $b_i = K(e_i, e_i)$, where $K(\cdot, \cdot)$ is the Killing form of $\gg$, and $e_i\in\n_i$. Also,  $m$ denotes the dimension of $\h^\perp$.
\begin{proposition}\label{propn:RFeqns}
 Let $\g(t)$ be a time-dependent diagonal cohomogeneity one metric evolving by the Ricci flow. Then the components $h, f_1, \cdots, f_m$ of $\g(t)$ satisfy the following system of PDEs:
\begin{equation}\label{eq:RFcohom1}
 \begin{aligned}
  &h_t = \sum_{j=1}^m \left(\frac{{f_j}_{rr}}{hf_j} - \frac{{f_j}_rh_r}{h^2f_j} \right)\\
  &{f_i}_t = \frac{{f_i}_{rr}}{h^2} - \frac{{f_i}_rh_r}{h^3} + \frac{{f_i}_r}{h}\sum_{j=1}^m \frac{{f_j}_r}{hf_j} - \frac{{f_i}_r^2}{h^2f_i} - \sum_{j,k=1}^m\frac{f_i^4-2f_k^4}{4f_if_j^2f_k^2}{\gamma_{jk}^i}^2 + \frac{b_i}{2f_i} \\
  &t\in (0,T),\, r\in(0,L),\, i = 1,\cdots m
 \end{aligned} 
\end{equation}
\end{proposition}
\begin{proof}
A time-dependent diagonal metric $\g$ and diagonal Ricci tensor can be written as:
\begin{align*}
 \g(r,t) &= h(r,t)^2\,\dd r^2 + \sum_{i=1}^m f_i(r,t)^2\, \omega_i^2\\
 \Ric_\g(r,t) &= \Ric_\g\left(\ddr, \ddr\right)\,\dd r^2 + \sum_{i=1}^m \Ric_\g(X_i^*, X_i^*)\,\omega_i^2
\end{align*}
Differentiating the metric term by term with respect to $t$ yields
\begin{align*}
 \frac{\dd\g}{\dd t} = 2hh_t\,\dd r^2 + \sum_{i=1}^m 2f_i{f_i}_t\, \omega_i^2
\end{align*}
On the other hand, by \cite[Proposition 1.14]{gz02} the Ricci tensor can be written in terms of the metric and the structure constants $\gamma_{ij}^k$ as
\begin{align}
 \Ric(e_i, e_i) &=  -\frac{b_i}{2} + \sum_{j,k=1}^m\frac{f_i^4-2f_k^4}{4f_j^2f_k^2}{\gamma_{jk}^i}^2 + \left\{ -\frac{{f_i}_r}{hf_i}\sum_{j=1}^m \frac{{f_j}_r}{hf_j} + \frac{{f_i}_r^2}{h^2f_i^2} - \frac{{f_i}_{rr}}{h^2f_i} + \frac{{f_i}_rh_r}{h^3f_i} \right\}f_i^2
\end{align}
Substituting these in the Ricci flow equation \eqref{eq:RF} and comparing coefficients then yields the result.
\end{proof}
We now use this system of PDEs to study the Ricci flow behavior of sectional curvature on a special class of cohomogeneity one manifolds.
\begin{theorem}\label{thm:sec_GZ}
 Let $M$ be a cohomogeneity one manifold with the action of $\SO(n)$ with a group diagram where the groups $\H$, $\K_\pm$ are products of $\SO(k)$ in  block embedding, and such that there are two singular orbits each of codimension two. Then $M$ admits a metric $\g$ such that $\sec_\g \geq 0$ and when evolved by the Ricci flow, $\g$ immediately acquires some negatively curved $2$-planes.
\end{theorem}
\begin{proof}
 Since the cohomogeneity one manifold ($M$, $\G$) has codimension $2$ singular orbits, by work of Grove and Ziller \cite{gz00}, $M$ admits a $\G$-invariant metric $\ggz$ with $\sec \geq 0$. By the construction in \cite{gz00}, one can arrange that the metric is diagonal in a basis coming from the standard basis vectors of $\mathfrak{so}(n)$. By Theorem \ref{mainthm:RF_SO}, the evolving metric $\g(t)$ will be diagonal in the same basis, and hence the components of the metric will satisfy \eqref{eq:RFcohom1}.
 
 By the Grove-Ziller construction, up to relabelling of indices, the functions $f_i$ that determine the metric have qualitative behaviour as follows. At a singular orbit, i.e. $r=0$, $f_1$ vanishes, and the remaining functions $f_i$ are equal and constant in a neighborhood of $r=0$. As a consequence, if we define $\mu(r)$ to be the $2$-plane spanned by $\ddr$ and $X_2$ then $\sec_{\ggz}\mu(r) = -\frac{f_2''}{f_2} = 0$ for $r$ close to $0$. Here $'$ denotes derivative with respect to arclength along $\gamma(r)$. We will compute the first variation of $\sec_{\g(t)}\mu(r)$ at $t=0$.
 Using the assumptions about the $f_i$'s in a neighborhood of $r=0$ for the metric $\g_{GZ}$,
 \begin{align*}
  \frac{\dd}{\dd t}\left(\frac{f_2''}{f_2}\right) = -\frac{({f_2})_{rrt}}{f_2}\big|_{t=0}
 \end{align*}
 By regularity of $f_2$, $({f_2})_{rrt} = (({f_2})_t)_{rr}$ which we can compute by twice differentiating with respect to $r$ the equation in \ref{eq:RFcohom1} corresponding to $f_2$. We compute this derivative at $t=0$ (i.e. for the metric $\ggz$) and for $r>0$ close to $0$. As a result, $f_i = c$ for $i>1$, for some constant $c$. Hence $({f_i})_r = ({f_i})_{rr} = 0$ for each $i>1$, so the expression for $({f_2})_t$ in a neighborhood of $r=0$ reduces to
 \begin{align*}
  ({f_2})_t\big|_{t=0} &= \frac{b_2}{2f_2} - \sum_{j,k=1}^m \frac{f_2^4 - 2f_k^4}{4f_2f_j^2f_k^2}(\gamma_{jk}^2)^2\\
  &= \frac{b_2}{2f_2} - \frac{f_2^4 - 2f_1^4}{4f_2f_3^2f_1^2}(\gamma_{31}^2)^2 - \frac{f_2^4 - 2f_3^4}{4f_2f_1^2f_3^2}(\gamma_{13}^2)^2 - \sum_{j,k \neq 1} \frac{f_2^4 - 2f_k^4}{4f_2f_j^2f_k^2}(\gamma_{jk}^2)^2
 \end{align*}
Without loss of generality we have assumed that $3$ is the unique index $j$ such that $\gamma_{1j}^2 = \gamma_{j1}^2$ is non-zero. That there is only one such index follows from the fact that $\{E_{ij}\}$ is a nice basis. (If the basis is not nice then there will be additional summands of the same form as the second and third summand in the above expression, with $3$ replaced by the suitable index $j$ for which $\gamma_{1j}^2 = -\gamma_{j1}^2$ is non-zero.) Therefore,
\begin{align*}
 ({f_2})_t\big|_{t=0} &= \frac{b_2}{2f_2} - \frac{2(f_2^4 - f_3^4) - 2f_1^4}{4f_2f_3^2f_1^2}(\gamma_{13}^2)^2 - \sum_{j,k \neq 1} \frac{f_2^4 - 2f_k^4}{4f_2f_j^2f_k^2}(\gamma_{jk}^2)^2\\
 &= \frac{b_2}{2c} + \frac{f_1^2}{2c^3}(\gamma_{13}^2)^2 + \sum_{j,k \neq 1} \frac{1}{4c}(\gamma_{jk}^2)^2\\
 \implies ({f_2})_{trr}\big|_{t=0} &= \frac{(\gamma_{13}^2)^2}{c^3}\cdot(({f_1})_r^2 + f_1({f_1})_{rr}) 
\end{align*}
By the smoothness conditions at a singular orbit, $f_1(r=0) = ({f_1})_{rr}(r=0) = 0$ and $({f_1})_r(r=0) = a$ for some $a \in \Z_+$. Therefore for small enough $r>0$, $({f_1})_r^2 + f_1({f_1})_{rr}>0$, hence $({f_2})_{trr}>0$ and $\frac{\dd}{\dd t}\sec(\mu(r))\big|_{t=0}<0$. We conclude that for small enough $t>0$, $\sec_{\g(t)}(\mu(r))<0$.
\end{proof}

\end{document}